\documentclass[reqno]{amsart}

\usepackage{amsfonts,amssymb,amsmath,latexsym,indentfirst,cite}

\usepackage{amsthm}
\usepackage{color,graphicx}
\newcommand{\R}{\mathbb R}

\newcommand{\tres}{|\!|\!|}

\newtheorem{theorem}{Theorem}[section]
\newtheorem{proposition}[theorem]{Proposition}
\newtheorem{remark}[theorem]{Remark}
\newtheorem{lemma}[theorem]{Lemma}
\newtheorem{corollary}[theorem]{Corollary}

\begin{document}
\title[Generalized KdV equation]{On well-posedness and wave operator for the gKdV equation}

%\author[]{}
%\address{}
%\email{}
%\thanks{}
%\begin{abstract}
%\end{abstract}

\maketitle

\begin{center}
 {\large \textbf{Luiz G. Farah\footnote{Partially supported by FAPEMIG, CNPq and CAPES/Brazil.}}}\\
 {\small ICEx, Universidade Federal de Minas Gerais \\
Av. Ant\^onio Carlos, 6627, Caixa Postal 702, 30123-970, Belo Horizonte-MG, Brazil.    \\
E-mail: lgfarah@gmail.com}\\
 \end{center}

  \begin{center}
  {\large \textbf{Ademir Pastor\footnote{Partially supported by FAPESP and CNPq/Brazil.}}}\\ {\small IMECC-UNICAMP\\
 Rua S\'ergio Buarque de Holanda, 651, 13083-859, Campinas-SP, Brazil.\\
  E-mail: apastor@ime.unicamp.br}\\
  \end{center}

 \begin{abstract}
We consider the generalized Korteweg-de Vries (gKdV) equation $\partial_t
u+\partial_x^3u+\mu\partial_x(u^{k+1})=0$, where $k>4$ is an integer number
and $\mu=\pm1$. We give an alternative proof of the Kenig, Ponce, and Vega
result in \cite{kpv1}, which asserts local and global well-posedness in
$\dot{H}^{s_k}(\R)$, with $s_k=(k-4)/2k$. A blow-up alternative in suitable
Strichatz-type spaces is also established. The main tool is a new linear
estimate. As a consequence, we also construct a wave operator in the critical
space $\dot{H}^{s_k}(\R)$, extending the results of C$\hat{\textrm{o}}$te
\cite{RC}.
 \end{abstract}

\section{Introduction}\label{introduction}

In this paper, we consider the  generalized Korteweg-de Vries (gKdV) equation
\begin{equation}\label{gkdv}
\begin{cases}
\partial_t u+\partial_x^3u+\mu\partial_x(u^{k+1})=0, \;\;x\in\R, \;t>0, \\
u(x,0)=u_0(x),
\end{cases}
\end{equation}
where $\mu=\pm1$ and $k > 4$ is an integer number.

In the particular case $k=1$, this equation was derived by Korteweg and de
Vries \cite{KdV95} in their study of waves on shallow water. Here, we are
mainly interested in the case $k>4$ (the $L^2$ supercritical case), which is
a generalization of the model proposed in \cite{KdV95}.

Well-posednes for the Cauchy problem \eqref{gkdv} is now well understood. We
first recall the scaling argument: if $u$ is a solution of \eqref{gkdv},
then, for any $\lambda>0$, $u_\lambda(x,t)=\lambda^{2/k}u(\lambda
x,\lambda^3t)$ is also a solution with initial data
$u_\lambda(x,0)=\lambda^{2/k}u_0(\lambda x)$. Moreover,
$$
\|u_\lambda(\cdot,0)\|_{\dot{H}^s}=\lambda^{s+2/k-1/2}\|u_0\|_{\dot{H}^s}.
$$
Thus, for each $k$ fixed, the scale-invariant Sobolev space is
$\dot{H}^{s_k}(\R)$, $s_k=1/2-2/k$. As a consequence, the natural Sobolev
spaces to study equation \eqref{gkdv} are  ${H}^{s}(\R)$ with $s\geq s_k$.
The well-posedness theory to the Cauchy problem \eqref{gkdv} was developed by
Kenig, Ponce, and Vega \cite{kpv1} (see also Kato \cite{K83} for a previous
result on this direction). Concerning the small data global theory in the
critical Sobolev space $\dot{H}^{s_k}(\mathbb{R})$, it was proved in
\cite[Theorem 2.15]{kpv1} the following result.

\begin{theorem}\label{local2}
Let $k>4$ and $s_k=(k-4)/2k$. Then there exists $\delta_k>0$ such that for any $u_0\in \dot{H}^{s_k}(\R)$ with
$$\|D^{s_k}_xu_0\|_{2}<\delta_k$$
there exists a unique solution $u(\cdot)$ of the IVP \eqref{gkdv} satisfying
\begin{equation}  \label{a12}
u\in C(\R:\dot{H}^{s_k}(\R)),
\end{equation}
\begin{equation}\label{a22}
\|D^{s_k}_xu\|_{L^5_xL^{10}_t}<\infty,
\end{equation}
\begin{equation}\label{a32}
\|D^{s_k}_xu_x\|_{L^\infty_xL^{2}_t}<\infty,
\end{equation}
and
\begin{equation}\label{a3b2}
\|D^{\alpha_k}_xD^{\beta_k}_tu\|_{L^{p_k}_xL^{q_k}_t}<\infty,
\end{equation}
where
\begin{equation} \label{a42}
\alpha_k=\frac{1}{10}-\frac{2}{5k}, \qquad
\beta_k=\frac{3}{10}-\frac{6}{5k}, \qquad
\end{equation}
\begin{equation} \label{a52}
\frac{1}{p_k}=\frac{2}{5k}+\frac{1}{10}, \qquad \frac{1}{q_k}=\frac{3}{10}-\frac{4}{5k}.
\end{equation}
Furthermore, the map $u_0\mapsto u(t)$ form $\{u_0\in
\dot{H}^{s_k}(\mathbb{R})\,:\, \|D^{s_k}_xu_0\|_{2}<\delta_k\}$ into the
class defined by \eqref{a12}-\eqref{a3b2} is Lipschitz.
\end{theorem}

The method to prove Theorem \ref{local2} combines smoothing effects and
Strichartz-type estimates together with the Banach contraction principle.
This result shows to be sharp in view of the work due to Birnir, Kenig,
Ponce, Svanstedt, and Vega \cite{BKPSV96}.

One of the main goals of this paper is to reprove the above theorem without
using any norm that involves derivative in the time variable. To this end,
we introduce a new linear estimate (see Lemma \ref{lemma12} below) which
allow us to obtain the following result.

\begin{theorem}\label{local}
Let $k\geq4$ and $s_k=(k-4)/2k$. Given $u_0\in \dot{H}^{s_k}(\R)$, assume
\begin{equation}\label{condK}
\|D^{s_k}_xu_0\|_{L^2_x}\leq K<\infty.
\end{equation}
There exists $\delta=\delta(K)>0$ such that if
\begin{equation}\label{condsmall}
\|U(t)u_0\|_{L^{5k/4}_xL^{5k/2}_t}<\delta,
\end{equation}
then there exists a unique solution $u$ of the integral equation
\begin{equation}\label{inteq}
u(t)=U(t)u_0-\mu\int_0^tU(t-t')\partial_x(u^{k+1})(t')dt'
\end{equation}
such that
\begin{equation}\label{i0}
\|u\|_{L^{5k/4}_xL^{5k/2}_t}\leq 2\delta \quad \textrm{and} \quad
\|D^{s_k}_xu\|_{L^5_xL^{10}_t}< 2cK.
\end{equation}
The solution also satisfies,
\begin{equation}\label{i1}
\|D^{s_k}_xu\|_{L^{\infty}_tL^2_x}<2cK.
\end{equation}
Moreover, there exist $f_\pm\in\dot{H}^{s_k}(\R)$ such that
\begin{equation}\label{scat}
\lim_{t\rightarrow\pm\infty}\|D^{s_k}_x(u(t)-U(t)f_\pm)\|_{L^2_x}=0.
\end{equation}
%Let $k>4$ and $s_k=(k-4)/2k$. Then there exists $\delta_k>0$ such that for any $u_0\in \dot{H}^{s_k}(\R)$ with
%$$\|D^{s_k}_xu_0\|_{2}<\delta_k$$
%there exists unique strong solution $u(\cdot)$ of the IVP \eqref{gkdv} satisfying
%\begin{equation}  \label{a1}
%u\in C(\R:\dot{H}^{s_k}(\R)),
%\end{equation}
%\begin{equation}\label{a2}
%\|D^{s_k}_xu\|_{L^5_xL^{10}_t}<\infty
%\end{equation}
%%\begin{equation}\label{a3}
%%\|D^{s_k}_xu_x\|_{L^\infty_xL^{2}_t}<\infty,
%%\end{equation}
%and
%\begin{equation}\label{a4}
%\|u\|_{L^{5k/4}_xL^{5k/2}_t}<\infty
%\end{equation}
%
%Furthermore, the map $u_0\mapsto u(t)$ form $\{u_0\in
%\dot{H}^{s_k}(\mathbb{R})\,:\, \|D^{s_k}_xu_0\|_{2}<\delta_k\}$ into the
%class defined by \eqref{a1}-\eqref{a4} is smooth.
\end{theorem}

It is worth to mention that the question of how small the initial data should
be to imply global well-posedness in the energy space $H^1(\R)$ have been
recently addressed by Farah, Linares, and Pastor \cite{FLP} where sufficient
conditions have been obtained.

Without imposing any smallness restriction on the initial data, a local
version of Theorem \ref{local2} is also available in \cite[Theorem
2.17]{kpv1}. Here, following the same strategy of Theorem \ref{local}, we
are able to prove the following local well-posedness result

\begin{theorem}\label{locald}
Let $k\geq4$ and $s_k=(k-4)/2k$. Given $u_0\in \dot{H}^{s_k}(\R)$ there exist $T=T(u_0)$ and a unique solution $u$ of the integral equation \eqref{inteq}
satisfying
\begin{equation}\label{d2}
\|D^{s_k}_xu\|_{L^{\infty}_{[0,T]}L^2_x}+\|u\|_{L^{5k/4}_xL^{5k/2}_{[0,T]}}+\|D^{s_k}_xu\|_{L^5_xL^{10}_{[0,T]}}<\infty.
\end{equation}

Furthermore, given $T'\in(0,T)$ there exists a neighborhood $V$ of $u_0$ in $\dot{H}^{s_k}(\R)$ such that the map $u_0\mapsto \tilde{u}(t)$ from $V$
into the class defined by \eqref{d2} in the time interval $[0,T']$ is Lipschitz.
\end{theorem}

Note that the previous result asserts that the existence time depends on the
initial data itself and not only on its norm. Our next theorem is concerned
with the behavior of the local solution near the possible blow-up time. This
is inspired by the results in \cite[Theorem 1.2]{KPV6}. Let $T^*=T^*(u_0)$ be
the maximum time of existence for the unique solution $u$ of the integral
equation \eqref{inteq} with initial data $u_0\in \dot{H}^{s_k}(\R)$. When
$T^*<\infty$, we prove that always
$\|u\|_{L^{5k/4}_xL^{5k/2}_{[0,T^{*}]}}=\infty$.  On the other hand, a direct
application of Theorem \ref{locald} implies that either the
$\dot{H}^{s_k}(\R)$-norm of $u(t)$ blows-up in time or the
$\dot{H}^{s_k}(\R)$-$\lim_{t\uparrow T^{\ast}}u(t)$ does not exist. However,
even in the case when the $\dot{H}^{s_k}(\R)$-norm of the solution $u(t)$
does not blow-up (this is the case when $k=4$ by the mass conservation law)
we stablished blow-up for the Strichartz norm that appears in \eqref{STR}. Our result reads as follows.
\begin{theorem}\label{BUR}
Assume $k\geq4$ and $s_k=(k-4)/2k$. Suppose $u_0\in \dot{H}^{s_k}(\R)$  and $T^*=T^*(u_0)$ be the maximum time of existence given in
Theorem \ref{locald}. If $T^*<\infty$ then
\begin{equation}\label{eq5}
\|u\|_{L^{5k/4}_xL^{5k/2}_{[0,T^{*}]}}=\infty. %, \quad for \;\; all\;\;t<T^*.
\end{equation}
Moreover, if $\sup_{t\in [0,T^{\ast})}\|D_x^{s_k}u(t)\|_{L^2_x}=K<\infty$, we also have
\begin{equation}\label{eq6}
\|D_x^{2/3k}u\|_{L^{3k/2}_xL^{3k/2}_{[0,T^*]}}=\infty. %, \quad for \;\; all\;\;t<T^*.
\end{equation}

\end{theorem}

\begin{remark}
In the limit case $k=4$ we
recover the result in \cite[Theorem 1.2]{KPV6}.
\end{remark}

Next we turn to the construction of the wave operator associated to
the equation \eqref{gkdv}. This is the reciprocal problem of the
scattering theory, which consists in constructing a solution with a
prescribed scattering state. Roughly speaking, for a given profile
$V$ (regardless of its size), one looks for a solution of the
nonlinear problem $u(t)$, defined for large enough $t$, such
that%
\[
\lim_{t\rightarrow\infty}\Vert u(t)-u_{V}(t)\Vert_{Y}=0,
\]
where $u_{V}(t)$ is the solution of the linear problem with initial data
$V$, and $Y$ stands for a suitable Banach space. Solving the latter problem
is also known as the construction of a wave operator.

This question was studied in Besov Spaces for the generalized Boussinesq
equation in \cite{farah} and in Sobolev or weighted Sobolev spaces for
Schr\"{o}dinger equations by Ginibre, Ozawa, and Velo \cite{Gi-Oz-Ve, GV} and
for generalized Korteweg-de Vries equations by C$\hat{\textrm{o}}$te
\cite{RC}. In this last paper, the author introduced two different approaches
to deal with this problem. The case $k>4$ was treated in the weighted Sobolev
space setting. On the other hand, the case $k=4$ ($L^2$-critical KdV
equation) is simpler since the fixed point problem is in fact very similar to
that of the Cauchy problem treated by Kenig, Ponce and Vega \cite{kpv1} in
their small data global existence theorem (see Theorem \ref{local2}). In this
paper, using the new linear estimate (Lemma \ref{lemma12} below) , we show that we can also apply
the same approach to the case $k>4$, extending the C$\hat{\textrm{o}}$te's result to the
classical Sobolev spaces. Our theorem is the following

\begin{theorem}\label{TIS}
Let $k\geq4$ and $s_k=(k-4)/2k$. For any $v\in \dot{H}^{s_k}(\R)$, let $u_{v}(t)$ be the solution of the linear problem associated
with \eqref{gkdv} with initial data $v$. Then, there exist $T_0=T_0(v)$ and
$u\in C([T_0,\infty ):\dot{H}^{s_k}(\R))$ solution of the IVP \eqref{gkdv} satisfying
\begin{equation}\label{LIM}
\lim_{t\rightarrow\infty}\Vert u(t)-u_{v}(t)\Vert_{\dot{H}^{s_k}(\R)}=0.
\end{equation}

Moreover, $u(\cdot)$ is unique in the class
$$\left\{u\in L^{\infty}_t\dot{H}^{s_k}(\R) /
D^{s_k}_xu \in L^5_xL^{10}_T,  u\in L^{5k/4}_xL^{5k/2}_T\right\}.$$
\end{theorem}

The plan of this paper is as follows. In the next section we introduce some notation and prove the linear estimates related to our problem.
In Section \ref{localwp} we prove Theorems \ref{local} and \ref{BUR}. Next, in Section \ref{IS}, we show Theorem \ref{TIS}.

%%%%%%%%%%%%%%%%%%%%%%%%%%%%%%%%%%%%%%%%%%%%%%%%%%%%%%%%%%%%%%%%%%%%%%%%%%%%%%%%%%%%%%%%%%%%%%%%%%%%%%%%%%%%%%%%%%%%%%%%%%%%
%%%%%%%%%%%%%%%%%%%%%%%%%%%%%%%%%%%%%%%%%%%%%%%%%%%%%%%%%%%%%%%%%%%%%%%%%%%%%%%%%%%%%%%%%%%%%%%%%%%%%%%%%%%%%%%%%%%%%%%%%%%%
%%%%%%%%%%%%%%%%%%%%%%%%%%%%%%%%%%%%%%%%%%%%%%%%%%%%%%%%%%%%%%%%%%%%%%%%%%%%%%%%%%%%%%%%%%%%%%%%%%%%%%%%%%%%%%%%%%%%%%%%%%%%
%%%%%%%%%%%%%%%%%%%%%%%%%%%%%%%%%%%%%%%%%%%%%%%%%%%%%%%%%%%%%%%%%%%%%%%%%%%%%%%%%%%%%%%%%%%%%%%%%%%%%%%%%%%%%%%%%%%%%%%%%%%%
%%%%%%%%%%%%%%%%%%%%%%%%%%%%%%%%%%%%%%%%%%%%%%%%%%%%%%%%%%%%%%%%%%%%%%%%%%%%%%%%%%%%%%%%%%%%%%%%%%%%%%%%%%%%%%%%%%%%%%%%%%%%
%%%%%%%%%%%%%%%%%%%%%%%%%%%%%%%%%%%%%%%%%%%%%%%%%%%%%%%%%%%%%%%%%%%%%%%%%%%%%%%%%%%%%%%%%%%%%%%%%%%%%%%%%%%%%%%%%%%%%%%%%%%%
%%%%%%%%%%%%%%%%%%%%%%%%%%%%%%%%%%%%%%%%%%%%%%%%%%%%%%%%%%%%%%%%%%%%%%%%%%%%%%%%%%%%%%%%%%%%%%%%%%%%%%%%%%%%%%%%%%%%%%%%%%%%
%%%%%%%%%%%%%%%%%%%%%%%%%%%%%%%%%%%%%%%%%%%%%%%%%%%%%%%%%%%%%%%%%%%%%%%%%%%%%%%%%%%%%%%%%%%%%%%%%%%%%%%%%%%%%%%%%%%%%%%%%%%%
%%%%%%%%%%%%%%%%%%%%%%%%%%%%%%%%%%%%%%%%%%%%%%%%%%%%%%%%%%%%%%%%%%%%%%%%%%%%%%%%%%%%%%%%%%%%%%%%%%%%%%%%%%%%%%%%%%%%%%%%%%%%

\section{Notation and Preliminaries}\label{notation}

Let us start this section by introducing the notation used throughout the paper. We use $c$ to denote various constants that may
vary line by line. Given any positive numbers $a$ and $b$, the notation $a \lesssim b$ means that there exists a positive constant
$c$ such that $a \leq cb$.

We use $\|\cdot\|_{L^p}$ to denote the $L^p(\R)$ norm. If necessary, we use subscript to inform which variable we are concerned with.
The mixed norms $L^q_tL^r_x$, $L^q_{[a,b]}L^r_x$ and ${L_{T}^{q}L^r_x}$ of $f=f(x,t)$ are defined, respectively, as
\begin{equation*}
\|f\|_{L^q_tL^r_x}= \left(\int_{-\infty}^{+\infty} \|f(\cdot,t)\|_{L^r_x}^q dt \right)^{1/q}\!\!\!\!, \quad \|f\|_{L^q_tL^r_x}= \left(\int_{a}^{b} \|f(\cdot,t)\|_{L^r_x}^q dt \right)^{1/q}
\end{equation*}
and
\begin{equation*}
\|f\|_{L^q_TL^r_x}= \left(\int_{T}^{+\infty} \|f(\cdot,t)\|_{L^r_x}^q dt \right)^{1/q}
\end{equation*}
with the usual modifications when $q =\infty$ or $r=\infty$.  Similarly, we also define the norms in the spaces  $L^r_xL^q_t$, $L^r_xL^q_{[a,b]}$ and ${L^r_x}L_{T}^{q}$.

The spatial Fourier transform of $f(x)$ is given by
\begin{equation*}
\hat{f}(\xi)=\int_{\R}e^{-ix\xi}f(x)dx.
\end{equation*}

% Recall that the derivative $\partial_x$ is conjugated to multiplication by $i\xi$ by the Fourier transform.

The class of Schwartz functions is represented by $\mathcal{S}(\R)$. We
shall also define $D_x^s$ and $J_x^s$ to be, respectively, the Fourier
multiplier with symbol $|\xi|^s$ and $\langle \xi \rangle^s = (1+|\xi|)^s$.
In this case, the norm in the Sobolev spaces $H^s(\R)$ and $\dot{H}^s(\R)$
are given, respectively, by
\begin{equation*}
\|f\|_{H^s}\equiv \|J_x^sf\|_{L^2_x}=\|\langle \xi
\rangle^s\hat{f}\|_{L^2_{\xi}}, \qquad \|f\|_{\dot{H}^s}\equiv
\|D_x^sf\|_{L^2_x}=\||\xi|^s\hat{f}\|_{L^2_{\xi}}.
\end{equation*}

Let us present now some useful lemmas and inequalities. In what follows, $U(t)$ denotes the linear propagator associated with the gKdV equation, that is, for any function $u_0$,  $U(t)u_0$ is the solution of the linear problem
\begin{equation}\label{lineargkdv}
\begin{cases}
\partial_t u+\partial_x^3u=0, \;\;x\in\R, \;t\in\R, \\
u(x,0)=u_0(x).
\end{cases}
\end{equation}

We begin by recalling the results necessary to prove Theorem \ref{local}.

\begin{lemma}\label{lemma1} Let $p,q$, and $\alpha$ be such that
\begin{equation}\label{adm}
-\alpha+\dfrac{1}{p}+\dfrac{3}{q}=\dfrac{1}{2}, \quad -\dfrac{1}{2}\leq\alpha\leq\dfrac{1}{q}.
\end{equation}
Then
\begin{equation}\label{reg1}
 \|D^{\alpha}_xU(t)u_0\|_{L_t^qL_x^p}\lesssim\|u_0\|_{L^2_x}.
\end{equation}
Also, if
\begin{equation}\label{eq3}
\frac{1}{p}+\frac{1}{2q}=\frac{1}{4}, \quad \alpha=\frac{2}{q}-\frac{1}{p},
\quad 1\leq p,q\leq\infty, \quad -\frac{1}{4}\leq\alpha\leq1.
\end{equation}
Then,
\begin{equation}\label{eq4}
\|D^{\alpha}_x U(t)u_0\|_{L^{p}_xL^{q}_t}\lesssim
\|u_0\|_{L^2_x}.
\end{equation}
\end{lemma}
\begin{proof}
See \cite[Theorem 2.1]{KPV4} and \cite[Proposition 2.1]{KPV6}.
\end{proof}

\begin{remark}
Note that when $(p,q)=(\infty, 2)$ in \eqref{eq4} we have the sharp version of the Kato smoothing effect
\begin{equation}\label{reg11}
 \|\partial_xU(t)u_0\|_{L_x^{\infty}L_t^2}\lesssim\|u_0\|_{L^2_x}.
\end{equation}
Moreover, the dual version of \eqref{reg11} reads as follows:
\begin{equation}\label{IE1}
\|\partial_x \int_0^tU(t-t')g(\cdot,t')dt'\|_{L^{2}_x}\lesssim
\| g\|_{L^{1}_xL^{2}_t}.
\end{equation}
(see \cite[Theorem 3.5]{kpv1}).
\end{remark}

\begin{lemma}\label{kpvlemma}
Assume
\begin{equation}\label{eq1}
\frac{1}{4}\leq\alpha<\frac{1}{2}, \qquad
\frac{1}{p}=\frac{1}{2}-\alpha.
\end{equation}
Then,
\begin{equation}\label{eq2}
\|D^{-\alpha}_x U(t)u_0\|_{L^{p}_xL^\infty_t}\lesssim
\|u_0\|_{L^2_x}.
\end{equation}
%Moreover
%\begin{equation}\label{reg11}
% \|\partial_xU(t)u_0\|_{L_x^{\infty}L_t^2}\lesssim\|u_0\|_{L^2_x}.
%\end{equation}
%In particular, the dual version of \eqref{reg11} reads as follows:
%\begin{equation}\label{IE1}
%\|\partial_x \int_0^tU(t-t')g(\cdot,t')dt'\|_{L^{2}_x}\lesssim
%\| g\|_{L^{1}_xL^{2}_t}.
%\end{equation}
\end{lemma}
\begin{proof}
See \cite[Lemma 3.29]{kpv1}. %and \cite[Theorem 3.5]{kpv1}.
\end{proof}

We can also obtain the following particular cases of the Strichartz estimates
in the critical Sobolev space $\dot{H}^{s_k}(\R)$.

\begin{corollary}\label{corollary12} Let $k\geq4$ and $s_k=(k-4)/2k$.
Then
\begin{equation}\label{STR2}
 \|D^{-1/k}_xU(t)u_0\|_{L^{k}_xL^{\infty}_t}\lesssim\|D^{s_k}u_0\|_{L^2_x}.
\end{equation}
and
\begin{equation}\label{STR}
 \|D^{2/3k}_xU(t)u_0\|_{L_x^{3k/2}L_t^{3k/2}}\lesssim\|D^{s_k}u_0\|_{L^2_x}.
\end{equation}
\end{corollary}
\begin{proof}
The estimate \eqref{STR2} is a particular case of \eqref{eq2}. On the other
hand, by Sobolev Embedding and inequality \eqref{reg1} with $q=3k/2$ and
$\alpha = 2/3k$, we obtain
\begin{equation*}
\begin{split}
\|D^{2/3k}_xU(t)u_0\|_{L_x^{3k/2}L_t^{3k/2}}\lesssim
\|D^{2/3k}_xU(t)D^{s_k}_xu_0\|_{L^{3k/2}_tL^{p}_x}
\lesssim\|D^{s_k}u_0\|_{L^2_x},
\end{split}
\end{equation*}
where $\dfrac1p=\dfrac12+\dfrac{2}{3k}-\dfrac{2}{k}$, which implies \eqref{STR}.
\end{proof}

Our next lemma is the fundamental tool to prove Theorem \ref{local}.
\begin{lemma}\label{lemma12} Let $k\geq4$ and $s_k=(k-4)/2k$.
Then
\begin{equation}\label{STR3}
 \|U(t)u_0\|_{L^{5k/4}_xL^{5k/2}_t}\lesssim\|D^{s_k}_xu_0\|_{L^2_x}.
\end{equation}
\end{lemma}
\begin{proof}
Interpolate the inequalities \eqref{STR} and \eqref{STR2}.
\end{proof}

Next, we recall the following integral estimates.
\begin{lemma}\label{lemma2}
If $p_i,q_i$, $\alpha_i$, $i=1,2$, satisfy \eqref{eq3}, then
$$
\|D_x^{\alpha_1} \int_0^tU(t-t')g(\cdot,t')dt'\|_{L^{p_1}_xL^{q_1}_t}\lesssim
\|D_x^{-\alpha_2} g\|_{L^{p_2'}_xL^{q_2'}_t},
$$
where $p_2'$ and $q_2'$ are the H\"older conjugate of $p_2$ and $q_2$,
respectively.
Moreover, assume $(p_1,q_1,\alpha_1)$ satisfies \eqref{eq3} and
$(p_2, \alpha_2)$ satisfies \eqref{eq1}. Then,
\begin{equation}\label{eq51}
    \|D^{-\alpha_2}\int_0^tU(t-t')g(\cdot,t')dt'\|_{L^{p_2}_xL^\infty_t}\lesssim
    \|D^{-\alpha_1}g\|_{L^{p_1'}_xL^{q_1'}_t}.
\end{equation}
\end{lemma}
\begin{proof}
Use the duality and $TT^*$ arguments combined with Lemmas
\ref{lemma1} and \ref{kpvlemma} (see also \cite[Proposition 2.2]{KPV6}).
\end{proof}

Applying the same ideas as the previous lemma together with Lemma \ref{lemma12} and \eqref{reg11} we also have.
\begin{corollary}\label{corollary21}
Let $k\geq4$ and $s_k=(k-4)/2k$, then the following estimate holds:
$$
\|\partial_x \int_0^tU(t-t')g(\cdot,t')dt'\|_{L^{5k/4}_xL^{5k/2}_t}\lesssim
\|D_x^{s_k} g\|_{L^{1}_xL^{2}_t}.
$$
\end{corollary}

\begin{remark}
In all of the above inequalities we can replace the integral $\int_{0}^{t}$
by $\int_{t}^{\infty}$. This kind of estimate will be used in Section
\ref{IS}.
\end{remark}

% \begin{lemma}\label{lemma5}
% Let $s\geq s_k$. Let $p_k$ and $q_k$ be as in Theorem \ref{local}. The following estimate is fulfilled
% \begin{equation*}
% \begin{split}
% \left\|D^s_x(u^k\partial_xu)\right\|_{L^{5/4}_xL^{10/9}_t}\lesssim &
% \left\|D^{\alpha_k}_xD^{\beta_k}_tu\right\|_{L^{p_k}_xL^{q_k}_t}^k\|D^s_x\partial_xu\|_{L^\infty_xL^2_t}\\
% &+ \left\|D^{\alpha_k}_xD^{\beta_k}_tu\right\|_{L^{p_k}_xL^{q_k}_t}^{k-1}\|u\|_{L^{5k/4}_xL^{5k/2}_t}
% \|D^s_x\partial_xu\|_{L^\infty_xL^2_t}.
% \end{split}
% \end{equation*}
% \end{lemma}
% \begin{proof}
% See proof of Proposition 6.1 in \cite{kpv1}.
% \end{proof}
Finally, we have the following estimates for fractional derivatives.
\begin{lemma}\label{lemma6}
Let $0<\alpha<1$ and $p,p_1,p_2,q,q_1,q_2 \in (1,\infty)$ with
$\frac{1}{p}=\frac{1}{p_1}+\frac{1}{p_2}$ and
$\frac{1}{q}=\frac{1}{q_1}+\frac{1}{q_2}$. Then,
\begin{itemize}
    \item[(i)] $$
\|D^\alpha_x(fg)-fD^\alpha_xg-gD^\alpha_xf\|_{L^p_xL^q_t}\lesssim
\|D^\alpha_xf\|_{L^{p_1}_xL^{q_1}_t}\|g\|_{L^{p_2}_xL^{q_2}_t}.
$$
The same still holds if $p=1$ and $q=2$.
    \item[(ii)] $$\|D^\alpha_xF(f)\|_{L^p_xL^q_t}\lesssim
    \|D^\alpha_xf\|_{L^{p_1}_xL^{q_1}_t}\|F'(f)\|_{L^{p_2}_xL^{q_2}_t}.$$
  \end{itemize}
\end{lemma}
\begin{proof}
See  \cite[Theorems A.6, A.8, and A.13]{kpv1}.\\
\end{proof}

%%%%%%%%%%%%%%%%%%%%%%%%%%%%%%%%%%%%%%%%%%%%%%%%%%%%%%%%%%%%%%%%%%%%%%%%%%%%%%%%%%%%%%%%%%%%%%%%%%%%%%%%%%%%%%%%%%%%%%%%%%%%
%%%%%%%%%%%%%%%%%%%%%%%%%%%%%%%%%%%%%%%%%%%%%%%%%%%%%%%%%%%%%%%%%%%%%%%%%%%%%%%%%%%%%%%%%%%%%%%%%%%%%%%%%%%%%%%%%%%%%%%%%%%%
%%%%%%%%%%%%%%%%%%%%%%%%%%%%%%%%%%%%%%%%%%%%%%%%%%%%%%%%%%%%%%%%%%%%%%%%%%%%%%%%%%%%%%%%%%%%%%%%%%%%%%%%%%%%%%%%%%%%%%%%%%%%
%%%%%%%%%%%%%%%%%%%%%%%%%%%%%%%%%%%%%%%%%%%%%%%%%%%%%%%%%%%%%%%%%%%%%%%%%%%%%%%%%%%%%%%%%%%%%%%%%%%%%%%%%%%%%%%%%%%%%%%%%%%%
%%%%%%%%%%%%%%%%%%%%%%%%%%%%%%%%%%%%%%%%%%%%%%%%%%%%%%%%%%%%%%%%%%%%%%%%%%%%%%%%%%%%%%%%%%%%%%%%%%%%%%%%%%%%%%%%%%%%%%%%%%%%
%%%%%%%%%%%%%%%%%%%%%%%%%%%%%%%%%%%%%%%%%%%%%%%%%%%%%%%%%%%%%%%%%%%%%%%%%%%%%%%%%%%%%%%%%%%%%%%%%%%%%%%%%%%%%%%%%%%%%%%%%%%%
%%%%%%%%%%%%%%%%%%%%%%%%%%%%%%%%%%%%%%%%%%%%%%%%%%%%%%%%%%%%%%%%%%%%%%%%%%%%%%%%%%%%%%%%%%%%%%%%%%%%%%%%%%%%%%%%%%%%%%%%%%%%
%%%%%%%%%%%%%%%%%%%%%%%%%%%%%%%%%%%%%%%%%%%%%%%%%%%%%%%%%%%%%%%%%%%%%%%%%%%%%%%%%%%%%%%%%%%%%%%%%%%%%%%%%%%%%%%%%%%%%%%%%%%%
%%%%%%%%%%%%%%%%%%%%%%%%%%%%%%%%%%%%%%%%%%%%%%%%%%%%%%%%%%%%%%%%%%%%%%%%%%%%%%%%%%%%%%%%%%%%%%%%%%%%%%%%%%%%%%%%%%%%%%%%%%%%

\section{Global well-posedness and the blow-up alternative}\label{localwp}

Our aim in this section is to establish Theorems \ref{local} and \ref{BUR}.

\begin{proof}[Proof of Theorem \ref{local}]
As usual, our proof is based on the contraction mapping principle. Hence, we
define
$$
X^k_{a,b}=\{ u\in C(\R;\dot{H}^{s_k}(\mathbb{R})):
\|u\|_{L^{5k/4}_xL^{5k/2}_t}\leq a  \quad \textrm{and} \quad
\|D^{s_k}_xu\|_{L^{\infty}_tL^2_x}+\|D^{s_k}_xu\|_{L^5_xL^{10}_t}\leq b\}.
$$
%$$
%X_k=\{u\in L^{5k/4}_xL^{5k/2}_t;\; \tres u\tres_{s_k}<\infty\}
%$$
%and
%$$
%X_k^a=\{u\in X_k;\;\tres u\tres_{s_k}\leq a\},
%$$
%where
%\begin{equation*}
%\begin{split}
%\tres
%u\tres_{s_k}=&\|u\|_{L^\infty_t\dot{H}^{s_k}}+\|D^{s_k}_xu\|_{L^5_xL^{10}_t}+\|u\|_{L^{5k/4}_xL^{5k/2}_t}.
%\end{split}
%\end{equation*}
On $X^k_{a,b}$ consider the integral operator
\begin{equation}\label{Phi}
\Phi(u)(t):=U(t)u_0-\mu\int_0^tU(t-t')\partial_x(u^{k+1})(t')dt'.
\end{equation}

We need to show that
$$\Phi:X^k_{a,b} \rightarrow X^k_{a,b}$$
is a contraction, for an appropriated choice of the parameters $a > 0$ and
$b>0$.

We first estimate the $\|\cdot\|_{L^{5k/4}_xL^{5k/2}_t}$--norm. From
Corollary \ref{corollary21},
\begin{equation*}
\begin{split}
\|\Phi(u)\|_{L^{5k/4}_xL^{5k/2}_t}&\leq
\|U(t)u_0\|_{L^{5k/4}_xL^{5k/2}_t}+c\|D^{s_k}_x(u^{k+1})\|_{L^1_xL^2_t}.
\end{split}
\end{equation*}
The estimate of the term $\|D^{s_k}_x(u^{k+1})\|_{L^1_xL^2_t}$ can be found
in \cite{kpv1} (see equation (6.1)). For the sake of completeness we will
also perform it here. Indeed, applying the fractional derivative rule  in
Lemma \ref{lemma6}-(i), we obtain
\begin{equation}\label{a6}
\begin{split}
\|D^{s_k}_x(u^{k+1})&\|_{L^1_xL^2_t}\lesssim
\|u^k\|_{L^{5/4}_xL^{5/2}_t}\|D^{s_k}_xu\|_{L^5_xL^{10}_t}
+\|uD^{s_k}_x(u^k)\|_{L^1_xL^2_t}\\
&\lesssim\|u\|_{L^{5k/4}_xL^{5k/2}_t}^k\|D^{s_k}_x u\|_{L^5_xL^{10}_t}
+\|u\|_{L^{5k/4}_xL^{5k/2}_t}\|D^{s_k}_x(u^k)\|_{L^{p_0}_xL^{q_0}_t}\\
&\lesssim\|u\|_{L^{5k/4}_xL^{5k/2}_t}^k\|D^{s_k}_x u\|_{L^5_xL^{10}_t}
+\|u\|_{L^{5k/4}_xL^{5k/2}_t}\|D^{s_k}_xu\|_{L^5_xL^{10}_t}\|u^{k-1}\|_{L^{p_1}_xL^{q_1}_t},
\end{split}
\end{equation}
where
$$
\frac{1}{p_1}\!=\!\frac{1}{p_0}-\frac{1}{5}\!=\!1-\frac{4}{5k}-\frac{1}{5}\!=\!\frac{4(k-1)}{5k}
\;\;\;
\text{and}
\;\;\;
\frac{1}{q_1}\!=\!\frac{1}{q_0}-\frac{1}{10}\!=\!\frac{1}{2}-\frac{2}{5k}-\frac{1}{10}\!=\!\frac{4(k-1)}{10k}.
$$

Therefore
\begin{equation}\label{a6a}
\begin{split}
\|\Phi(u)\|_{L^{5k/4}_xL^{5k/2}_t}&\leq \|U(t)u_0\|_{L^{5k/4}_xL^{5k/2}_t} + ca^kb\\
&\leq \delta +ca^kb.
\end{split}
\end{equation}

On the other hand, using the inhomogeneous  smoothing effect \eqref{IE1}, Lemma
\ref{lemma2} with $(p_1,q_1,\alpha_1)=(5,10,0)$,
$(p_2,q_2,\alpha_2)=(\infty,2,1)$, and \eqref{eq4} with
$(p,q,\alpha)=(5,10,0)$, estimate \eqref{a6} also implies
\begin{equation}\label{a6b}
\begin{split}
\|D^{s_k}_x\Phi(u)\|_{L^{\infty}_tL^2_x}+\|D_x^{s_k}\Phi(u)\|_{L^{5}_xL^{10}_t}&\leq cK+c\|D^{s_k}_xU(t)u_0\|_{L^5_xL^{10}_t}+c\|D^{s_k}_x(u^{k+1})\|_{L^1_xL^2_t}\\
&\leq cK +ca^kb.
\end{split}
\end{equation}

First one chooses $b=2cK$ and $a$ such that $ca^k\leq1/2$.
Therefore, once for all
$$
\|D^{s_k}_x\Phi(u)\|_{L^{\infty}_tL^2_x}+\|D^{s_k}_x\Phi(u)\|_{L^5_xL^{10}_t}\leq b.
$$
Finally,  if one chooses $\delta=a/2$ and $a$ so that
$ca^{k-1}b\leq1/2$ then we also have
$$
\|\Phi(u)\|_{L^{5k/4}_xL^{5k/2}_t}\leq a.
$$
Such calculations establish that $\Phi:X^k_{a,b} \rightarrow X^k_{a,b}$ is
well defined. For the contraction one uses similar arguments. The contraction
mapping principle then imply the existence of a unique fixed point for
$\Phi$, which is a solution of \eqref{inteq}. The proof is completed with
standard arguments.

To prove \eqref{scat}, one defines
$$
f_\pm=u_0+\mu\int_0^{\pm\infty}U(-t')\partial_x(u^{k+1})(t')dt'.
$$
Then, as in \eqref{a6}-\eqref{a6b}, we have
$$
\|D^{s_k}_x(u(t)-U(t)f_+)\|_{L^2_x}\leq\|D^{s_k}_x\int_t^{+\infty}
U(t-t')\partial_x(u^{k+1})(t')dt'\|_{L^2_x}\lesssim
\|u\|^k_{L^{5k/4}_xL^{5k/2}_{[t,+\infty)}}\|D^{s_k}_xu\|_{L^5_xL^{10}_{[t,+\infty)}}.
$$
Since $\|u\|_{L^{5k/4}_xL^{5k/2}_t}<\infty$ and
$\|D^{s_k}_xu\|_{L^5_xL^{10}_t}<\infty$, the above inequality implies
\eqref{scat}. A similar calculation holds for $f_-$.
\end{proof}

Theorem \ref{locald} follows from similar
arguments, so it will be omitted. Next, we prove the blow-up result
stated in Theorem \ref{BUR}.

\begin{proof}[Proof of Theorem \ref{BUR}]
The proof is based on the arguments in \cite[Theorem 1.2]{KPV6}. We first claim that if
$\|u\|_{L^{5k/4}_xL^{5k/2}_{[0,T^*]}}<\infty$ then $\|D^{s_k}_xu\|_{L^{5}_xL^{10}_{[0,T^*]}}<\infty$. Indeed, let
$\varepsilon_0>0$ be an arbitrary small number. Since the norm $\|u\|_{L^{5k/4}_xL^{5k/2}_{[0,T^*]}}$ is finite we can split the interval $[0,T^*]$
in  $0=t_0<t_1<\ldots<t_\ell=T^*$ such that
$\|u\|_{L^{5k/4}_xL^{5k/2}_{I_n}}<\varepsilon_0$, where $I_n=[t_n,t_{n+1}]$,
$n=0,1,\dots,\ell-1$. Since $u$ is a fixed point of the integral equation
\eqref{Phi}, from \eqref{eq4} with $(p,q,\alpha)=(5,10,0)$, we have, for
$n=0,1,\dots,\ell-1$,
\begin{equation}\label{eq9.1}
\begin{split}
\|D_x^{s_k}u\|_{L^5_xL^{10}_{I_n}}&\lesssim
\|D_x^{s_k}u_0\|_{L^2_x}+\|D^{s_k}\int_0^tU(t-t')\partial_x(u^{k+1})(t')dt'\|_{L^5_xL^{10}_{I_n}}\\
&\lesssim\|D_x^{s_k}u_0\|_{L^2_x}+\sum_{j=0}^n\|D_x^{s_k}\int_{t_j}^{t_{j+1}}U(t-t')\partial_x(u^{k+1})(t')dt'\|_{L^5_xL^{10}_{I_n}}\\
&\lesssim \|D_x^{s_k}u_0\|_{L^2_x}+\sum_{j=0}^n\|D_x^{s_k}\int_{0}^{t}U(t-t')\partial_x(u^{k+1})(t')\chi_{I_j}(t')dt'\|_{L^5_xL^{10}_{t}},\\
\end{split}
\end{equation}
where $\chi_{I_j}$ denotes the characteristic function of the interval $I_j$.
By using Lemma \ref{lemma2}, with $(p_1,q_1,\alpha_1)=(5,10,0)$ and
$(p_2,q_2,\alpha_2)=(\infty,2,1)$, we then deduce
\begin{equation}\label{eq9.2}
\begin{split}
\|D_x^{s_k}u\|_{L^5_xL^{10}_{I_n}}\lesssim
\|D_x^{s_k}u_0\|_{L^2_x}+\sum_{j=0}^n\|D_x^{s_k}(u^{k+1})\|_{L^1_xL^{2}_{I_j}}.\\
\end{split}
\end{equation}
Now, inequality \eqref{a6} yields
\begin{equation}\label{eq9.3}
\begin{split}
\|D_x^{s_k}u\|_{L^5_xL^{10}_{I_n}}&\lesssim
\|D_x^{s_k}u_0\|_{L^2_x}+\sum_{j=0}^n\|D_x^{s_k}u\|_{L^5_xL^{10}_{I_j}}\|u\|_{L^{5k/4}_xL^{5k/2}_{I_j}}^k\\
&\leq
c\|D_x^{s_k}u_0\|_{L^2_x}+c\varepsilon_0^k\sum_{j=0}^n\|D_x^{s_k}u\|_{L^5_xL^{10}_{I_j}}.
\end{split}
\end{equation}
Therefore, choosing $c\varepsilon_0^k<1/2$, we conclude
\begin{equation}\label{eq9.3a}
\begin{split}
\|D_x^{s_k}u\|_{L^5_xL^{10}_{I_n}}\leq
2c\|D_x^{s_k}u_0\|_{L^2_x}+2\sum_{j=0}^{n-1}\|D_x^{s_k}u\|_{L^5_xL^{10}_{I_j}}.
\end{split}
\end{equation}
Inequality \eqref{eq9.3a} together with an induction argument implies that
$\|D_x^{s_k}u\|_{L^5_xL^{10}_{I_n}}<\infty$ for $n=0,1,\dots,\ell-1$. By summing
over the $\ell$ intervals we conclude the claim.
%that $\|D_x^{s_k}u\|_{L^{5}_xL^{10}_{[0,T^*]}}<\infty$.

Next we prove \eqref{eq5}. Assume that the solution exists for $|t|<T'$ with $\|u\|_{L^{5k/4}_xL^{5k/2}_{[0,T']}}<\infty$, by the above claim we also have $\|D^{s_k}_xu\|_{L^5_xL^{10}_{[0,T']}}<\infty$. Moreover, by inequality \eqref{a6}, we conclude for all $t\in [0,T']$
$$
\|D_x^{s_k}u(t)\|_{L^2_x}\lesssim
\|D_x^{s_k}u_0\|_{L^2}+\|u\|_{L^{5k/4}_xL^{5k/2}_{[0,T']}}^k\|D_x^{s_k}u\|_{L^{5}_xL^{10}_{[0,T']}}<\infty.
$$
As a consequence, $u(T')\in \dot{H}^{s_k}(\mathbb{R})$, which from Theorem
\ref{locald} implies the existence of $\delta>0$ such that the
solution exists for $|t|\leq
T'+\delta$. Thus, if $T^*<\infty$, \eqref{eq5} must be true.\\

Next, we turn to proof of  \eqref{eq6}. Let $\delta>0$ be a small constant to be chosen later. Since
$\|u\|_{L^{5k/4}_xL^{5k/2}_{[0,T^*]}}=\infty$, we can construct a family of intervals $I_n=[t_n,t_{n+1}]$ such that
$t_n<t_{n+1}$, $t_n\nearrow T^*$, and
\begin{equation}\label{eq10.0}
\|u\|_{L^{5k/4}_xL^{5k/2}_{I_n}}=\delta.
\end{equation}
From analytic interpolation, we obtain
\begin{equation}\label{eq10}
\|u\|_{L^{5k/4}_xL^{5k/2}_{I_n}}\lesssim
\|D_x^{2/3k}u\|_{L^{3k/2}_xL^{3k/2}_{I_n}}^{3/5}\|D_x^{-1/k}u\|_{L^k_xL^\infty_{I_n}}^{2/5}.
\end{equation}
On the other hand, since $u$ is a fixed point of integral equation \eqref{Phi}, for $t\in I_n$ we also have
\begin{equation}\label{IEMOD}
u(t)=U(t-t_n)u(t_n)-\mu\int_{t_n}^tU(t-t')\partial_x(u^{k+1})(t')dt'.
\end{equation}
Therefore
$$
\|D_x^{-1/k}u\|_{L^k_xL^\infty_{I_n}}\lesssim
\|D_x^{-1/k}U(t-t_n)u(t_n)\|_{L^k_xL^\infty_{I_n}}+\|D_x^{-1/k}\int_{t_n}^tU(t-t')\partial_x(u^{k+1})dt'\|_{L^k_xL^\infty_{I_n}}.
$$
To bound the linear part we use Corollary \ref{corollary12} and to
bound the integral part, we use Lemma \ref{lemma2}-\eqref{eq51} with
$(p_2,\alpha_2)=(k,1/2-1/k)$ and $(p_1,q_1,\alpha_1)=(\infty, 2,1)$.
Hence,
\begin{equation*}
\begin{split}
\|D_x^{-1/k}u\|_{L^k_xL^\infty_{I_n}}&\lesssim
\|D_x^{s_k}U(-t_n)u(t_n)\|_{L^2_x}+\|D_x^{s_k}(u^{k+1})\|_{L^1_xL^2_{I_n}}\\
&\leq cK+c\delta^k\|D_x^{s_k}u\|_{L^{5}_xL^{10}_{I_n}},
\end{split}
\end{equation*}
where in the last inequality we have used $\sup_{t\in [0,T^{\ast})}\|D_x^{s_k}u(t)\|_{L^2_x}=K$ and \eqref{eq10.0}.

It remains to bound the norm $\|D_x^{s_k}u\|_{L^{5}_xL^{10}_{I_n}}$. Indeed using again the integral equation \eqref{IEMOD}, Lemma \ref{lemma2} with $(p_1,q_1,\alpha_1)=(5,10,0)$, $(p_2,q_2,\alpha_2)=(\infty,2,1)$ and estimate \eqref{a6} we obtain
\begin{equation}\label{eq9.300}
\begin{split}
\|D_x^{s_k}u\|_{L^5_xL^{10}_{I_n}}&\lesssim
\|D_x^{s_k}U(-t_n)u(t_n)\|_{L^2_x}+\|D_x^{s_k}u\|_{L^5_xL^{10}_{I_n}}\|u\|_{L^{5k/4}_xL^{5k/2}_{I_n}}^k\\
&\leq cK+c\delta^k\|D_x^{s_k}u\|_{L^5_xL^{10}_{I_n}}.
\end{split}
\end{equation}

Hence, choosing $c\delta^k<1/2$ we conclude $\|D_x^{s_k}u\|_{L^5_xL^{10}_{I_n}}\leq 2cK$, which also implies $\|D_x^{-1/k}u\|_{L^k_xL^\infty_{I_n}}\leq 2cK$. The last two inequalities combining with \eqref{eq10}  yield
$$
\|D_x^{2/3k}u\|_{L^{3k/2}_xL^{3k/2}_{I_n}}\geq
\dfrac{\delta^{5/3}}{(2cK)^{2/3}},\quad \!\!\!\! \textrm{for all } \quad
\!\!\!\! n\in \mathbb{N}
$$
and, therefore \eqref{eq6} holds.
\end{proof}

%%%%%%%%%%%%%%%%%%%%%%%%%%%%%%%%%%%%%%%%%%%%%%%%%%%%%%%%%%%%%%%%%%%%%%%%%%%%%%%%%%%%%%%%%%%%%%%%%%%%%%%%%%%%%%%%%%%%%%%%%%%%
%%%%%%%%%%%%%%%%%%%%%%%%%%%%%%%%%%%%%%%%%%%%%%%%%%%%%%%%%%%%%%%%%%%%%%%%%%%%%%%%%%%%%%%%%%%%%%%%%%%%%%%%%%%%%%%%%%%%%%%%%%%%
%%%%%%%%%%%%%%%%%%%%%%%%%%%%%%%%%%%%%%%%%%%%%%%%%%%%%%%%%%%%%%%%%%%%%%%%%%%%%%%%%%%%%%%%%%%%%%%%%%%%%%%%%%%%%%%%%%%%%%%%%%%%
%%%%%%%%%%%%%%%%%%%%%%%%%%%%%%%%%%%%%%%%%%%%%%%%%%%%%%%%%%%%%%%%%%%%%%%%%%%%%%%%%%%%%%%%%%%%%%%%%%%%%%%%%%%%%%%%%%%%%%%%%%%%
%%%%%%%%%%%%%%%%%%%%%%%%%%%%%%%%%%%%%%%%%%%%%%%%%%%%%%%%%%%%%%%%%%%%%%%%%%%%%%%%%%%%%%%%%%%%%%%%%%%%%%%%%%%%%%%%%%%%%%%%%%%%
%%%%%%%%%%%%%%%%%%%%%%%%%%%%%%%%%%%%%%%%%%%%%%%%%%%%%%%%%%%%%%%%%%%%%%%%%%%%%%%%%%%%%%%%%%%%%%%%%%%%%%%%%%%%%%%%%%%%%%%%%%%%
%%%%%%%%%%%%%%%%%%%%%%%%%%%%%%%%%%%%%%%%%%%%%%%%%%%%%%%%%%%%%%%%%%%%%%%%%%%%%%%%%%%%%%%%%%%%%%%%%%%%%%%%%%%%%%%%%%%%%%%%%%%%
%%%%%%%%%%%%%%%%%%%%%%%%%%%%%%%%%%%%%%%%%%%%%%%%%%%%%%%%%%%%%%%%%%%%%%%%%%%%%%%%%%%%%%%%%%%%%%%%%%%%%%%%%%%%%%%%%%%%%%%%%%%%
%%%%%%%%%%%%%%%%%%%%%%%%%%%%%%%%%%%%%%%%%%%%%%%%%%%%%%%%%%%%%%%%%%%%%%%%%%%%%%%%%%%%%%%%%%%%%%%%%%%%%%%%%%%%%%%%%%%%%%%%%%%%

\section{The construction of the wave operator}\label{IS}

In this section, we intend to show Theorem \ref{TIS}. Following the ideas
introduced by C\^ote \cite{RC}, we must look for a fixed point for the
operator
\[
\Phi:{w}(t)\longrightarrow-\mu \partial_x \int_{t}^{\infty}U(t-t^{\prime})(
{w}(t^{\prime})+U(t^{\prime})v)^{k+1}dt^{\prime},
\]
defined in the time interval $[T_{0},\infty)$, where $T_0>0$ is an
arbitrarily large number that will be chosen later.

The next proposition says that this fixed point provides a function ${u}$ satisfying the integral equation (\ref{gkdv}) (we refer to Farah \cite{farah} for a proof).

\begin{proposition}
\label{prop1} Let $w$ be a fixed point of the operator $\Phi$ and
define
\begin{equation}
\label{DOU}{u} (t) = U(t)v+{w}(t).
\end{equation}
Then ${u} $ is a solution of (\ref{gkdv}) in the time interval
$[T_{0},\infty)$.
\end{proposition}

\begin{proof}[Proof of Theorem \ref{TIS}] Again we use the contraction
mapping principle. Given $T>0$, define the metric spaces
$$
X_T=\{w\in C(\R;\dot{H}^{s_k}(\mathbb{R}));\;\tres w\tres_{X_T}<\infty\}
$$
and
$$
X_T^a=\{w\in X_T;\;\tres w\tres_{X_T}\leq a\},
$$
where
\begin{equation*}
\begin{split}
\tres w \tres_{X_T}=&\|w\|_{L^\infty_T\dot{H}^{s_k}}+\|D^{s_k}_xw\|_{L^5_xL^{10}_T}+\|w\|_{L^{5k/4}_xL^{5k/2}_T}.
\end{split}
\end{equation*}
Applying Corollary \ref{corollary21} we obtain
\begin{equation*}
\begin{split}
\|\Phi(w)\|_{L^{5k/4}_xL^{5k/2}_T}&\lesssim \|D^{s_k}_x(U(t)v+{w}(t))^{k+1})\|_{L^1_xL^2_T}.
\end{split}
\end{equation*}
Using the same arguments as the ones used in \eqref{a6}, we conclude
\begin{equation*}
\|D^{s_k}_x(U(t)v+{w}(t))^{k+1})\|_{L^1_xL^2_T}\lesssim \|U(t)v+{w}(t)\|_{L^{5k/4}_xL^{5k/2}_T}^k\|D^{s_k}_x (U(t)v+{w}(t))\|_{L^5_xL^{10}_T}.
\end{equation*}

The other norms can be estimated in the same manner, which implies
\begin{equation}\label{a61}
\begin{split}
\|\Phi(w)\|_{X_T}&\lesssim \|U(t)v+{w}(t)\|_{L^{5k/4}_xL^{5k/2}_T}^k\|D^{s_k}_x (U(t)v+{w}(t))\|_{L^5_xL^{10}_T}\\
&\lesssim \|U(t)v+{w}(t)\|_{L^{5k/4}_xL^{5k/2}_T}^k\|D^{s_k}_x (U(t)v+{w}(t))\|_{L^5_xL^{10}_T}\\
&\lesssim \|U(t)v\|_{L^{5k/4}_xL^{5k/2}_T}^k\|D^{s_k}_x v\|_{L^2_x}+
\|D^{s_k}_x v\|_{L^2_x}\|w\|_{X_T}^{k}+\|w\|_{X_T}^{k+1},
\end{split}
\end{equation}
where in the last inequality we have used Lemma \ref{lemma1} and Lemma \ref{lemma12}.

Since $\|U(t)v\|_{L^{5k/4}_xL^{5k/2}_T}\rightarrow 0$ as $T\rightarrow \infty$ we can find a $T_0>0$ large enough and $a>0$ small enough such that $\Phi:X_{T_0}^a\rightarrow X_{T_0}^a$ is well defined and is a contraction. Therefore, $\Phi$ has a unique fixed point, which we denote by ${w}$. Moreover, $a>0$ can be chosen such that
\begin{equation}\label{a612}
\begin{split}
\|w\|_{X_T}&\lesssim  \|U(t)v\|_{L^{5k/4}_xL^{5k/2}_T}^k\|D^{s_k}_x v\|_{L^2_x},
\end{split}
\end{equation}
which implies $\|w\|_{X_T}\rightarrow 0$ as $T\rightarrow \infty$.

Next we show the limit (\ref{LIM}). By the Proposition, ${u} (t) = U(t)v+{w}(t)$ satisfies the integral equation in the right hand side of \eqref{Phi} in the time interval $[T_{0},\infty)$. Therefore
\begin{equation*}
\begin{split}
\|u(t)-U(t)v\|_{L^\infty_T\dot{H}^{s_k}}&= \|w\|_{L^\infty_T\dot{H}^{s_k}}= \|\Phi(w)\|_{L^\infty_T\dot{H}^{s_k}}\\
&\lesssim \|U(t)v\|_{L^{5k/4}_xL^{5k/2}_T}^k\|D^{s_k}_x v\|_{L^2_x}+
\|D^{s_k}_x v\|_{L^2_x}\|w\|_{X_T}^{k}+\|w\|_{X_T}^{k+1}.
\end{split}
\end{equation*}
The last inequality together with \eqref{a612} implies \eqref{LIM}, finishing the proof of the theorem.

%Finally, we deal with the uniqueness assertion. Let $\vec{u},\vec{v}%
%\in\mathcal{E}_{\sigma}^{T_{0}}$ be two solutions satisfying (\ref{mild1}) and
%corresponding to the same profile $\vec{h}.$ Write $\vec{w}_{1}(t)=\vec
%{u}(t)-B(t)\vec{h}$ and $\vec{w}_{2}(t)=\vec{v}(t)-B(t)\vec{h}$. It follows
%from (\ref{aux441}) that%
%\begin{align}
%\Vert\vec{w}_{1}-\vec{w}_{2}\Vert_{\mathcal{E}_{\sigma}^{T}}  &  =\Vert
%\Phi(\vec{w}_{1})-\Phi(\vec{w}_{2})\Vert_{\mathcal{E}_{\sigma}^{T}}\nonumber\\
%&  \leq CI_{\sigma}(T)\Vert\vec{w}_{1}-\vec{w}_{2}\Vert_{\mathcal{E}_{\sigma
%}^{T}}(\,\Vert\vec{u}\Vert_{\mathcal{E}_{\sigma}^{T}}+\Vert\vec{v}%
%\Vert_{\mathcal{E}_{\sigma}^{T}})^{\rho-1}, \label{uni1}%
%\end{align}
%for $T\geq T_{0}.$ Let $b=\max\{\Vert\vec{u}\Vert_{\mathcal{E}_{\sigma}%
%^{T_{0}}},\Vert\vec{v}\Vert_{\mathcal{E}_{\sigma}^{T_{0}}}\}$ and take $T\geq
%T_{0}$ such that
%\begin{equation}
%CI_{\sigma}(T)2^{\rho-1}\,b^{\rho-1}\leq1/2. \label{uni2}%
%\end{equation}
%In view of (\ref{uni1})-(\ref{uni2}) we obtain
%\[
%\Vert\vec{w}_{1}-\vec{w}_{2}\Vert_{\mathcal{E}_{\sigma}^{T}}\leq\frac{1}%
%{2}\Vert\vec{w}_{1}-\vec{w}_{2}\Vert_{\mathcal{E}_{\sigma}^{T}},
%\]
%which implies $\vec{w}_{1}(t)=\vec{w}_{2}(t)$ in $[T,\infty),$ and so $\vec
%{u}(t)=\vec{v}(t)$ in $[T,\infty).$
\end{proof}

%%%%%%%%%%%%%%%%%%%%%%%%%%%%%%%%%%%%%%%%%%%%%%%%%%%%%%%%%%%%%%%%%%%%%%%%%%%%%%%%%%%%%%%%%%%%%%%%%%%%%%%%%%%%%%%%%%%%%%%%%%%%
%%%%%%%%%%%%%%%%%%%%%%%%%%%%%%%%%%%%%%%%%%%%%%%%%%%%%%%%%%%%%%%%%%%%%%%%%%%%%%%%%%%%%%%%%%%%%%%%%%%%%%%%%%%%%%%%%%%%%%%%%%%%
%%%%%%%%%%%%%%%%%%%%%%%%%%%%%%%%%%%%%%%%%%%%%%%%%%%%%%%%%%%%%%%%%%%%%%%%%%%%%%%%%%%%%%%%%%%%%%%%%%%%%%%%%%%%%%%%%%%%%%%%%%%%
%%%%%%%%%%%%%%%%%%%%%%%%%%%%%%%%%%%%%%%%%%%%%%%%%%%%%%%%%%%%%%%%%%%%%%%%%%%%%%%%%%%%%%%%%%%%%%%%%%%%%%%%%%%%%%%%%%%%%%%%%%%%
%%%%%%%%%%%%%%%%%%%%%%%%%%%%%%%%%%%%%%%%%%%%%%%%%%%%%%%%%%%%%%%%%%%%%%%%%%%%%%%%%%%%%%%%%%%%%%%%%%%%%%%%%%%%%%%%%%%%%%%%%%%%
%%%%%%%%%%%%%%%%%%%%%%%%%%%%%%%%%%%%%%%%%%%%%%%%%%%%%%%%%%%%%%%%%%%%%%%%%%%%%%%%%%%%%%%%%%%%%%%%%%%%%%%%%%%%%%%%%%%%%%%%%%%%
%%%%%%%%%%%%%%%%%%%%%%%%%%%%%%%%%%%%%%%%%%%%%%%%%%%%%%%%%%%%%%%%%%%%%%%%%%%%%%%%%%%%%%%%%%%%%%%%%%%%%%%%%%%%%%%%%%%%%%%%%%%%
%%%%%%%%%%%%%%%%%%%%%%%%%%%%%%%%%%%%%%%%%%%%%%%%%%%%%%%%%%%%%%%%%%%%%%%%%%%%%%%%%%%%%%%%%%%%%%%%%%%%%%%%%%%%%%%%%%%%%%%%%%%%
%%%%%%%%%%%%%%%%%%%%%%%%%%%%%%%%%%%%%%%%%%%%%%%%%%%%%%%%%%%%%%%%%%%%%%%%%%%%%%%%%%%%%%%%%%%%%%%%%%%%%%%%%%%%%%%%%%%%%%%%%%%%

\section*{Acknowledgements}
%%acknowledgements here%%\subsection*{Acknowledgment}
The authors thank Felipe Linares and an anonymous referee for their helpful comments and suggestions which improved the presentation of the paper.

\bibliographystyle{amsplain}

\end{document}